\documentclass[a4paper,12pt]{amsart}
\usepackage[margin=3truecm,top=4truecm,bottom=4truecm]{geometry}

\usepackage{amsthm,amsmath,amssymb,comment,setspace,mleftright,constants,cite}
\usepackage[dvipdfmx]{graphicx,color}

\allowdisplaybreaks 

\usepackage{lineno}
\newcommand*\patchAmsMathEnvironmentForLineno[1]{%
  \expandafter\let\csname old#1\expandafter\endcsname\csname #1\endcsname
  \expandafter\let\csname oldend#1\expandafter\endcsname\csname end#1\endcsname
  \renewenvironment{#1}%
     {\linenomath\csname old#1\endcsname}%
     {\csname oldend#1\endcsname\endlinenomath}}%
\newcommand*\patchBothAmsMathEnvironmentsForLineno[1]{%
  \patchAmsMathEnvironmentForLineno{#1}%
  \patchAmsMathEnvironmentForLineno{#1*}}%
\AtBeginDocument{%
\patchBothAmsMathEnvironmentsForLineno{equation}%
\patchBothAmsMathEnvironmentsForLineno{align}%
\patchBothAmsMathEnvironmentsForLineno{flalign}%
\patchBothAmsMathEnvironmentsForLineno{alignat}%
\patchBothAmsMathEnvironmentsForLineno{gather}%
\patchBothAmsMathEnvironmentsForLineno{multline}%
}

\makeatletter

\@addtoreset{equation}{section}
\makeatother

\theoremstyle{plain} 
\newtheorem{thm}{Theorem}[section]
\newtheorem{prop}[thm]{Proposition}
\newtheorem{cor}[thm]{Corollary}
\newtheorem{lem}[thm]{Lemma}

\theoremstyle{definition}

\theoremstyle{remark}
\newtheorem{rem}[thm]{Remark}

\newcommand{\N}{\mathbb{N}}
\newcommand{\R}{\mathbb{R}}
\newcommand{\Z}{\mathbb{Z}}
\newcommand{\Q}{\mathbb{Q}}

\renewcommand{\P}{\mathbb{P}}
\newcommand{\E}{\mathbb{E}}

\newcommand{\1}[1]{\mathbf{1}_{#1}}

\renewcommand{\tilde}{\widetilde}

\renewcommand{\bar}{\overline}

\newcommand{\hyphen}{\textrm{-}}
\newcommand{\as}{\textrm{a.s.}}

\renewcommand{\S}{\mathbb{S}}

\begin{document}
\title[The simple random walk on supercritical percolation clusters]
	{Continuity result for the rate function\\ of the simple random walk on supercritical percolation clusters}
\author[N.~KUBOTA]{Naoki KUBOTA}
\address[N. Kubota]
	{College of Science and Technology, Nihon University, Chiba 274-8501, Japan.}
\email{kubota.naoki08@nihon-u.ac.jp}
\thanks{The author was supported by JSPS Grant-in-Aid for Young Scientists (B) 16K17620.}
\date{\today}

\begin{abstract}
We consider the simple random walk on supercritical percolation clusters in the multidimensional cubic lattice.
In this model, a quenched large deviation principle holds for the position of the random walk.
Its rate function depends on the law of the percolation configuration, and the aim of this paper is to study the continuity of the rate function in the law.
To do this, it is useful that the rate function is expressed by the so-called Lyapunov exponent, which is the asymptotic cost paid by the random walk for traveling in a landscape of percolation configurations.
In this context, we first observe the continuity of the Lyapunov exponent in the law of the percolation configuration, and then lift it to the rate function. 
\end{abstract}

\maketitle

\section{Introduction}

\subsection{The model}
For $d \geq 2$, we denote by $\Z^d$ the $d$-dimensional cubic lattice.
Furthermore, $\mathcal{E}^d$ is the set of all nearest-neighbor edges in $\Z^d$, i.e.,
\begin{align*}
	\mathcal{E}^d:=\bigl\{ \{ x,y \} \subset \Z^d: \|x-y\|_1=1 \bigr\},
\end{align*}
where $\|\cdot\|_1$ is the $\ell^1$-norm on $\R^d$.
Let $\omega=(\omega(e))_{e \in \mathcal{E}^d}$ denote a family of independent random variables  satisfying
\begin{align*}
	\P_p(\omega(e)=1)=1-\P_p(\omega(e)=0)=p \in [0,1].
\end{align*}
An edge $e \in \mathcal{E}^d$ is called \emph{open} if $\omega(e)=1$, and \emph{closed} otherwise.
We say that a lattice path is open if it uses only open edges.
Then, the \emph{chemical distance} $d(x,y)=d_\omega(x,y)$ between $x$ and $y$ is defined by the minimal length of an open lattice path from $x$ to $y$ in the percolation configuration $\omega$.

For $x \in \Z^d$, we denote by $\mathcal{C}_x=\mathcal{C}_x(\omega)$ the \emph{open cluster} containing $x$, i.e., the set of all vertices which are linked to $x$ by an open lattice path.
It is well known that there exists $p_c=p_c(d) \in (0,1)$ such that $\P_p$-almost surely, we have a unique infinite open cluster $\mathcal{C}_\infty=\mathcal{C}_\infty(\omega)$ with $\P_p(0 \in \mathcal{C}_\infty)>0$ whenever $p \in (p_c,1]$ (see Theorems~1.10 and 8.1 of \cite{Gri99_book} for instance).

Let $\mathcal{O}=\mathcal{O}(\omega)$ be the set of all vertices which are endpoints of open edges.
Define for $x,y \in \mathcal{O}$ with $\|x-y\|_1=1$,
\begin{align*}
	\pi_\omega(x,y):=\frac{\1{\{ \omega(\{x,y\})=1 \}}}{\sum_{\{x,y'\} \in \mathcal{E}^d}\1{\{ \omega(\{x,y'\})=1 \}}}
	\in \Bigl[ \frac{1}{2d},1 \Bigr].
\end{align*}
Then, the (discrete time) \emph{simple random walk on percolation clusters} is the Markov chain $((X_n)_{n=0}^\infty,(P_\omega^z)_{z \in \mathcal{O}})$ on $\mathcal{O}$ with the transition probabilities
\begin{align*}
	&P_\omega^z(X_0=z)=1,\\
	&P_\omega^z (X_{n+1}=y|X_n=x)=\pi_\omega(x,y).
\end{align*}

For each $x \in \Z^d$, denote by $H(x)$ the first passage time through $x$, i.e.,
\begin{align*}
	H(x):=\inf\{ n \geq0:X_n=x\}.
\end{align*}
Then, for any $\lambda \geq 0$ and $x,y \in \mathcal{O}$, we define the \emph{travel cost} $a_\lambda(x,y)=a_\lambda(x,y,\omega)$ from $x$ to $y$ as
\begin{align*}
	a_\lambda(x,y)=a_\lambda(x,y,\omega):=-\log e_\lambda(x,y,\omega),
\end{align*}
where 
\begin{align*}
	e_\lambda(x,y,\omega):=E_\omega^x\bigl[ e^{-\lambda H(y)}\1{\{ H(y)<\infty \}} \bigr].
\end{align*}
Obviously, the travel cost $a_\lambda(x,y)$ is dominated by the chemical distance $d(x,y)$ as follows:
\begin{align}\label{eq:chemi}
	d(x,y) \leq a_\lambda(x,y) \leq (\lambda+\log(2d))d(x,y),\qquad x,y \in \mathcal{O}.
\end{align}
Furthermore, the strong Markov property gives the triangle inequality
\begin{align}\label{eq:triangle}
	a_\lambda(x,z) \leq a_\lambda(x,y)+a_\lambda(y,z),\qquad x,y,z \in \mathcal{O},
\end{align}
see Lemma~{2.1} of \cite{Kub12b} for the proof.
Roughly speaking, these inequalities enable us to use the subadditive ergodic theorem for the travel cost, and we can derive the following asymptotic behavior.

\begin{prop}[{\cite[Theorem~{1.2}]{Kub12b}}]\label{prop:K_lyap}
Let $p \in (p_c,1]$ and $\lambda \geq 0$.
There exists a norm $\alpha_\lambda^p(\cdot)$ on $\R^d$ (which is called the Lyapunov exponent) such that the following hold:
\begin{itemize}
	\item $\P_p \hyphen \as$ on $\{ 0 \in \mathcal{C}_\infty \}$,
		\begin{align}\label{eq:lyap_def}
			\lim_{\substack{\|x\|_1 \to \infty\\ x \in \mathcal{C}_\infty}} \frac{a_\lambda(0,x)-\alpha_\lambda^p(x)}{\|x\|_1}=0.
		\end{align}
	\item The norm $\alpha_\lambda^p(\cdot)$ is invariant under permutations of the coordinates and under reflections in the coordinate hyperplanes,
		and has the bounds
		\begin{align}\label{eq:lyap_bd}
			\lambda\|x\|_1 \leq \alpha_\lambda^p(x) \leq \alpha_\lambda^p(\xi_1)\|x\|_1,
		\end{align}
		where $\xi_1$ is the first coordinate vector of $\R^d$.
\end{itemize} 
In particular, $\alpha_\lambda^p(x)$ is concave increasing in $\lambda$ and convex in $x$.
Moreover, it is jointly continuous in $\lambda$ and $x$.
\end{prop}

The Lyapunov exponent above plays an important role in deriving large deviations for the simple random walk on supercritical percolation clusters.

\begin{prop}[{\cite[Theorem~{1.3}]{Kub12b}}]\label{prop:K_ldp}
Let $p \in (p_c,1]$ and set
\begin{align}\label{eq:rate_def}
	I^p(x):=\sup_{\lambda \geq 0}(\alpha_\lambda^p(x)-\lambda),\qquad x \in \R^d.
\end{align}
Then, $\P_p \hyphen \as$ on the event $\{ 0 \in \mathcal{C}_\infty \}$, the law of the scaled random walk $X_n/n$ obeys the following large deviation principle with the rate function $I^p$:
\begin{itemize}
	\item (Upper bound) For any closed subset $A$ of $\R^d$,
		\begin{align*}
			\limsup_{n \to \infty} \frac{1}{n} \log P_\omega^0(X_n \in nA) \leq -\inf_{x \in A} I^p(x).
		\end{align*}
	\item (Lower bound) For any open subset $B$ of $\R^d$,
		\begin{align*}
			\liminf_{n \to \infty} \frac{1}{n} \log P_\omega^0(X_n \in nB) \geq -\inf_{x \in B} I^p(x).
		\end{align*}
\end{itemize}
\end{prop}

\begin{rem}\label{rem:Berger}
Berger et al.~\cite{BerMukOka16} obtained a large deviation result for the distribution of the empirical measures of the environment Markov chain of the simple random walk on percolation clusters.
This obeys a higher level large deviation principle than that of Proposition~\ref{prop:K_ldp}.
Therefore, via a contraction principle, Proposition~\ref{prop:K_ldp} is derived with another expression of the rate function, which is built on a certain relative entropy functional.
\end{rem}

\subsection{Main results}\label{subsect:main}
Our main results are the following continuities for the Lyapunov exponent $\alpha_\lambda^p(x)$ and the rate function $I^p(x)$ in $p$:

\begin{thm}\label{thm:lyap_conti}
Let $\lambda>0$ and $x \in \R^d$.
If $p$ and $p_n$, $n \geq 1$, belong to $(p_c,1]$ and $p_n \to p$ as $n \to \infty$, then
\begin{align}\label{eq:lyap_conti}
	\lim_{n \to \infty} \alpha_\lambda^{p_n}(x)=\alpha_\lambda^p(x).
\end{align}
\end{thm}

\begin{thm}\label{thm:rate_conti}
Denote by $\mathcal{D}_{I^p}$ the effective domain of the rate function $I^p$, i.e.,
\begin{align*}
	\mathcal{D}_{I^p}:=\{ y \in \R^d:I^p(y)<\infty \}.
\end{align*}
Then, let $x$ be in the interior $(\mathcal{D}_{I^p})^o$ of $\mathcal{D}_{I^p}$.
If $p$ and $p_n$, $n \geq 1$, belong to $(p_c,1]$ and $p_n \to p$ as $n \to \infty$, then
\begin{align*}
	\lim_{n \to \infty} I^{p_n}(x)=I^p(x).
\end{align*}
\end{thm}

\begin{rem}
In this paper, we treat the simple random walk on ``bond'' percolation clusters.
As mentioned in \cite[Remark~1.4]{Kub12b}, the same problems above are considered in the ``site'' percolation case.
To prove Theorems~\ref{thm:lyap_conti} and \ref{thm:rate_conti}, we use several estimates for bond percolation (see Subsections~\ref{subsect:est_perc} and \ref{subsect:coupling}).
It is known that these estimates are valid for site percolation, and Theorems~\ref{thm:lyap_conti} and \ref{thm:rate_conti} similarly hold in the site percolation case.
\end{rem}

Large deviations are recently interesting topics in random walks in random environments.
There are two different situations on this subject: the quenched (almost sure) and the annealed (in average) situations.
(The first one deals with time evolution of the random walk on a fixed realization of the random environment.
On the other hand, in the second one, we first average the randomness of the environment before letting the time grow.)
In the one-dimensional case, both quenched and annealed large deviations, including properties of these rate functions, have been studied well.
We refer the reader to \cite{GanZei98,GredHol94,PisPov99} and \cite{DemPerZei96,PisPovZei99} for the quenched and the annealed cases, respectively.
See also \cite[Subsections~2.4 and 2.5]{Zei04_book} as a survey of results on the aforementioned articles.
Furthermore, Comets--Gantert--Zeitouni~\cite{ComGanZei00} obtained a relation between the quenched and the annealed rate functions by solving a variational problem.
The one-dimensional large deviation principle remains attractive and its study is still in progress (see for instance~\cite{AhnPet16,Pet08_thesis,Pet10}).

For the multidimensional case, Zerner~\cite{Zer98b} first proved a large deviation principle for the random walk in i.i.d.~nestling environments.
After that, Varadhan~\cite{Var03} not only developed it to general ergodic environments, but also showed the annealed large deviation principle in i.i.d.~environments.
These results were later generalized by Rassoul-Agha et~al.~\cite{RA04,RASep11,RASep15_book,RASepYil13}.
Moreover, Peterson et al.~\cite{PetZei09b,Yil11,YilZei10} recently studied some properties of the quenched and the annealed rate functions.
However, we still have less information about these rate functions than is known in the one-dimensional case.
Furthermore, all these works need some moment assumptions for random environments.
In particular, those are satisfied if we have ellipticity, i.e., the random walk can always move to any nearest-neighbor site.
The random conductance model (which is a specific class of random walks in random environments) often treats the non-elliptic case, see \cite{Kum14_book} and the references given there for details.
Its typical case is the simple random walk on percolation clusters.
For this reason, the study of this large deviation principle has not progressed well.
As mentioned in Proposition~\ref{prop:K_ldp} and Remark~\ref{rem:Berger}, we succeeded in overcoming the lack of ellipticity by using some classical results about the geometry of percolation clusters, but still do not have enough information for the rate function.
Therefore, Theorem~\ref{thm:rate_conti} is meaningful in the investigation of rate functions for random walks in random environments.

Finally, let us comment on earlier works related to our results.
Theorem~\ref{thm:lyap_conti} plays the key role of the proof of Theorem~\ref{thm:rate_conti}, and similar continuity results to Theorem~\ref{thm:lyap_conti} have already been studied in the first passage percolation, the simple random walk in random potentials and the directed polymer in random environments. (See \cite{Cox80,CoxKes81,GarMarProThe17,ComFukNakYos15,Nak18,Le17} for details.)
For the counterpart of the Lyapunov exponent in each model above, we can easily derive the upper semi-continuity as a direct consequence of the subadditive ergodic theorem.
On the other hand, the proof of the lower semi-continuity is relatively difficult.
To this end, we have (at least) the following two approaches:
\begin{enumerate}
	\item Derive a lower large deviation estimate for the travel cost and combine it with a renormalization argument.
	\item Obtain a concentration inequality for the travel cost and estimate the difference between the expectation of the travel cost and the counterpart
		of the Lyapunov exponent (We call that kind of estimate the non-random fluctuation).
\end{enumerate}
In our model, the second approach may work, but it is hard to take constants appearing in the concentration inequality and the non-random fluctuation uniformly in the law of the percolation configuration.
This is critical to our work because we have to treat several different laws of percolation configurations simultaneously.
Hence, to prove Theorem~\ref{thm:lyap_conti}, the first approach is adapted in this paper.
For our desired lower large deviation (see Proposition~\ref{prop:lld} below), it is useful to apply the strategy which was taken by Ahlberg~\cite[Section~3]{Ahl15} for the first passage percolation.
However, we have to modify it due to the following difference between the first passage percolation and our model:
In the first passage percolation, we can take at least one optimal path realizing the travel cost under several conditions and only have to observe the lower large deviation along an optimal path.
On the other hand, in our model, the travel cost is defined by the average over all random walk trajectories and we cannot fix a random walk trajectory realizing the travel cost.

\subsection{Organization of the paper}\label{subsec:org}
Let us now describe how the present article is organized.
In Section~\ref{sect:prelim}, we summarize results of percolation on $\Z^d$ for our convenience.
Subsection~\ref{subsect:est_perc} provides some estimates for the chemical distance and the open clusters.
Moreover, we recall stochastic domination between locally dependent fields and the independent Bernoulli site percolation.
In Subsection~\ref{subsect:coupling}, a coupling is built to treat several percolation configurations simultaneously.
In Subsection~\ref{subsect:modigy}, we introduce a modification of the travel cost, which measures the cost of traveling from a hole to another one in the infinite cluster.

The goal of Section~\ref{sect:lyap_conti} is to prove Theorem~\ref{thm:lyap_conti}.
We divide the proof into two parts: the upper and the lower semi-continuities, which are stated in Subsections~\ref{subsect:lyap_u} and \ref{subsect:lyap_l}, respectively.
The upper semi-continuity is a direct consequence of the subadditive ergodic theorem for the modified travel cost.
On the other hand, the proof of the lower semi-continuity is more difficult than that of the upper semi-continuity.
This difficulty comes from the reasons stated at Subsection~\ref{subsect:main}:
Handling several different laws of percolation configurations at the same time and the non-existence of an optimal path realizing the travel cost.

The aim of Section~\ref{sect:rate_conti} is to show Theorem~\ref{thm:rate_conti}.
To this end, we first determine the effective domain of the rate function.
Actually, it can be described by the so-called time constant for the chemical distance, and in Subsection~\ref{subsect:relation}, we verify this claim via the relation between the Lyapunov exponent and the time constant for the chemical distance.
The proof of Theorem~\ref{thm:rate_conti} is given in Subsection~\ref{subsect:rate_conti}.
Its main tools are the concavity of the Lyapunov exponent and Theorem~\ref{thm:lyap_conti}.

We close this section with some general notation.
Write $\|\cdot\|_1$ and $\|\cdot\|_\infty$ for the $\ell^1$ and $\ell^\infty$-norms on $\R^d$.
In addition, for $i \in \{1,\infty\}$, $x \in \R^d$ and $r>0$, $B_i(x,r)$ is the $\ell^i$-ball in $\R^d$ of center $x$ and radius $r$, i.e.,
\begin{align*}
	B_i(x,r):=\{ y \in \R^d:\|y-x\|_i \leq r \}.
\end{align*}
Moreover, the notation $\S^{d-1}$ means the $\ell^1$-unit sphere.
Throughout this paper, we use $c$, $c'$, $C$, $C'$ and $C_i$, $i=1,2,\dots$, to denote constants with $0<c,c',C,C',C_i<\infty$.
In the present article, several percolation configurations are dealt with at the same time.
Hence, we often use the notation $C_i^{(p)}$ to emphasize dependence only on a parameter $p$.

\section{Preliminaries}\label{sect:prelim}

\subsection{Some estimates for percolation}\label{subsect:est_perc}
The following proposition presents estimates for the chemical distance and the open clusters.

\begin{prop}[{\cite[(2.2) and Corollary~2.2]{GarMar10}} and {\cite[Theorem~8.65]{Gri99_book}}]\label{prop:GM}
For $p \in (p_c,1]$, there exist constants $\Cl{AP} \geq 1$, $\Cl{GM1}$ and $\Cl{GM2}$ such that the following results (i)--(iii) hold: 
\begin{enumerate}
	\item  For all $x \in \Z^d$ and $t \geq \Cr{AP}\|x\|_1$,
		\begin{align*}
			\P_p(t \leq d(0,x)<\infty) \leq \Cr{GM1}e^{-\Cr{GM2}t}.
		\end{align*}
	\item For all $t \geq 0$,
		\begin{align*}
			\P_p(\mathcal{C}_\infty \cap B_1(0,t)=\emptyset) \leq \Cr{GM1}e^{-\Cr{GM2}t}.
		\end{align*}
	\item For all $t \geq 0$ and $x \in \Z^d$,
		\begin{align*}
			\P_p(t \leq \#\mathcal{C}_x<\infty) \leq \Cr{GM1}e^{-\Cr{GM2}t^{1-1/d}},
		\end{align*}
		where $\# A$ denotes the cardinality of a set $A$.
\end{enumerate}
\end{prop}

A part of the proof of Theorem~\ref{thm:lyap_conti} relies on a renormalization argument.
Hence, it is convenient to recall the concept of stochastic domination.
Let $Y=(Y_v)_{v \in \Z^d}$ and $Z=(Z_v)_{v \in \Z^d}$ be families of random variables taking values in $\{0,1\}$.
We say that $Y$ \emph{stochastically dominates} $Z$ if
\begin{align*}
	\E[f(Y)] \geq \E[f(Z)]
\end{align*}
for all bounded, increasing, measurable functions $f:\{0,1\}^{\Z^d} \to \R$.
Furthermore, a family $Y=(Y_v)_{v \in \Z^d}$ of random variables is said to be \emph{finitely dependent} if there exists $N>0$ such that any two sub-families $(Y_{v_1})_{v_1 \in \Lambda_1}$ and $(Y_{v_2})_{v_2 \in \Lambda_2}$ are independent whenever $\Lambda_1,\Lambda_2 \subset \Z^d$ satisfies that $\|v_1-v_2\|_1>N$ for all $v_1 \in \Lambda_1$ and $v_2 \in \Lambda_2$.

Under the preparation above, we state a stochastic domination for locally dependent fields with the independent Bernoulli site percolation.

\begin{prop}[{\cite[Theorem~7.65]{Gri99_book} or \cite[Theorem~B26]{Lig99_book}}]\label{prop:domi}
Suppose that $Y=(Y_v)_{v \in \Z^d}$ is a finitely dependent family of random variables taking values in $\{ 0,1 \}$.
For a given $\kappa \in (0,1)$, $Y$ stochastically dominates the independent Bernoulli site percolation $\eta_\kappa=(\eta_\kappa(v))_{v \in \Z^d}$ of parameter $\kappa$, provided $\inf_{v \in \Z^d} P(Y_v=1)$ is sufficiently close to one.
\end{prop}

\subsection{Coupling of percolation configurations}\label{subsect:coupling}
To treat several different Lyapunov exponents simultaneously, it is useful to introduce a coupling of percolation configurations.
To do this, for $p \in [0,1]$, we write $F_p$ for the distribution function of $\P_p(\omega(e) \in \cdot)$, i.e.,
\begin{align*}
	F_p(t):=\P_p(\omega(e) \leq t),\qquad t \in \R.
\end{align*}
Independently of $\omega$, let $(U(e))_{e \in \mathcal{E}^d}$ be independent random variables with the uniform distribution on $(0,1)$.
We now define for $p \in [0,1]$,
\begin{align*}
	\omega_p(e):=F_p^{-1}(U(e)),\qquad e \in \mathcal{E}^d,
\end{align*}
where $F_p^{-1}$ is the pseudo-inverse function of $F_p$:
\begin{align*}
	F_p^{-1}(s):=\sup\{ t \in \R:F_p(t)<s \},\qquad s \in \R,
\end{align*}
with the convention $\sup\emptyset:=0$.
It is well known that this coupling has the following properties:
\begin{itemize}
	\item The random variables $\omega_p=(\omega_p(e))_{e \in \mathcal{E}^d}$ are independent and identically distributed
		with $\P(\omega_p(e)=1)=1-\P(\omega_p(e)=0)=p$.
	\item Suppose that $p$ and $p_n$, $n \geq 1$, belong to $[0,1]$.
		If $p_n \to p$ as $n \to \infty$, then for all $e \in \mathcal{E}^d$, $\lim_{n \to \infty} \omega_{p_n}(e)=\omega_p(e)$ holds almost surely.
		In particular, for each $e \in \mathcal{E}^d$, with probability one, $\omega_{p_n}(e)$ coincides with $\omega_p(e)$, provided $n$ is large enough.
\end{itemize}

\subsection{Modification of the travel cost}\label{subsect:modigy}
For $q \in (p_c,1]$ and $x \in \Z^d$, let $[x]_q$ denote the closest point to $x$ in $\mathcal{C}_\infty(\omega_q)$ for the $\ell^1$-norm, with a deterministic rule to break ties.
Let $p_c<q \leq p \leq 1$.
We now define for $\lambda \geq 0$ and $x,y \in \Z^d$,
\begin{align*}
	a_\lambda^q(x,y,\omega_p):=a_\lambda([x]_q,[y]_q,\omega_p).
\end{align*}
The triangle inequality is inherited from the original travel cost:
\begin{align*}
	a_\lambda^q(x,z,\omega_p) \leq a_\lambda^q(x,y,\omega_p)+a_\lambda^q(y,z,\omega_p),\qquad x,y,z \in \Z^d.
\end{align*}

The aim of this subsection is to show the following proposition.

\begin{prop}\label{prop:modify}
Let $p_c<q \leq p \leq 1$ and $\lambda \geq 0$.
For all $x \in \Z^d$, almost surely and in $L^1$,
\begin{align*}
	\alpha_\lambda^p(x)
	&= \lim_{k \to \infty} \frac{1}{k} a_\lambda^q(0,kx,\omega_p)\\
	&= \lim_{k \to \infty} \frac{1}{k} \E[a_\lambda^q(0,kx,\omega_p)]
		= \inf_{k \geq 1} \frac{1}{k} \E[a_\lambda^q(0,kx,\omega_p)].
\end{align*}
In particular, we have almost surely,
\begin{align*}
	\lim_{\|x\|_1 \to \infty} \frac{a_\lambda^q(0,x,\omega_p)-\alpha_\lambda^p(x)}{\|x\|_1}=0.
\end{align*}
\end{prop}

For the proof, let us first derive some estimates for the chemical distance from a hole to another one in the infinite cluster.

\begin{lem}\label{lem:ible}
Let $q \in (p_c,1]$ and $\gamma>0$.
There exists a constant $\Cl{ible}=\Cr{ible}^{(q)}$ such that for all $x \in \Z^d$,
\begin{align*}
	\E[d_{\omega_q}([0]_q,[x]_q)^\gamma ] \leq \Cr{ible}\|x\|_1^\gamma.
\end{align*}
In particular, we have for all $p \in [q,1]$, $\lambda \geq 0$ and $x \in \R^d$,
\begin{align}\label{eq:a_ible}
	\alpha_\lambda^p(x) \leq \Cr{ible}^{(q)}(\lambda+\log(2d))\|x\|_1.
\end{align}
\end{lem}
\begin{proof}
Set $c:=3\Cr{AP}$ and use the translation invariance of $\omega_q$ to obtain that for all $t \geq 0$,
\begin{align}\label{eq:a_tail}
\begin{split}
	\P(d_{\omega_q}([0]_q,[x]_q) \geq t)
	&\leq 2\P(\mathcal{C}_\infty(\omega_q) \cap B_1(0,t/c)=\emptyset)\\
	&\quad +\sum_{\substack{y \in B_1(0,t/c)\\ z \in B_1(x,t/c)}} \P(t \leq d_{\omega_q}(y,z)<\infty).
\end{split}
\end{align}
Proposition~\ref{prop:GM}-(ii) yields that the first term of the right side in \eqref{eq:a_tail} is not larger than $2\Cr{GM1}e^{-\Cr{GM2}t/c}$.
On the other hand, since $\Cr{AP}\|y-z\|_1 \leq t$ for all $t \geq c\|x\|_1$, $y \in B_1(0,t/c)$ and $z \in B_1(x,t/c)$, Proposition~\ref{prop:GM}-(i) proves that there exists a constant $c'$ such that for all $t \geq c\|x\|_1$, the second term of the right side in \eqref{eq:a_tail} is smaller than or equal to $c't^{2d} e^{-\Cr{GM2}t}$.
With these observations,
\begin{align*}
	\E[d_{\omega_q}([0]_q,[x]_q)^\gamma]
	&= \int_0^\infty \P\bigl( d_{\omega_q}([0]_q,[x]_q) \geq t^{1/\gamma} \bigr)\,dt\\
	&\leq (c\|x\|_1)^\gamma+\int_{(c\|x\|_1)^\gamma}^\infty \P\bigl( d_{\omega_q}([0]_q,[x]_q) \geq t^{1/\gamma} \bigr)\,dt\\
	&\leq (c\|x\|_1)^\gamma+\int_0^\infty \bigl( 2\Cr{GM1}e^{-\Cr{GM2}t/c}+c't^{2d} e^{-\Cr{GM2}t} \bigr)\,dt.
\end{align*}
The last integral is bounded uniformly in $\|x\|_1$, and the first assertion follows.

Next observe \eqref{eq:a_ible}.
Since $\alpha_\lambda^p(\cdot)$ is a norm on $\R^d$, it suffices to check \eqref{eq:a_ible} for $x \in \Z^d$.
We use \eqref{eq:chemi} and Proposition~\ref{prop:modify} to obtain
\begin{align*}
	\alpha_\lambda^p(x) \leq \E[a_\lambda^q(0,x,\omega_p)]
	\leq (\lambda+\log(2d))\E[d_{\omega_p}([0]_q,[x]_q)].
\end{align*}
Therefore, \eqref{eq:a_ible} follows from the first assertion.
\end{proof}

\begin{lem}\label{lem:maximal}
Let $p_c<q \leq 1$.
Then, there exist constants $\Cl{maximal1}=\Cr{maximal1}^{(q)}$ and $\Cl{maximal2}=\Cr{maximal2}^{(q)}$ such that for all $\epsilon>0$ and $x \in \Z^d$,
\begin{align}\label{eq:maximal}
\begin{split}
	&\P\Bigl( \sup\{ d_{\omega_q}([x]_q,[y]_q):y \in \Z^d,\,\|x-y\|_1<\epsilon\|x\|_1 \} \geq 3\Cr{AP}^{(q)}\epsilon\|x\|_1 \Bigr)\\
	&\leq \Cr{maximal1}e^{-\Cr{maximal2}\epsilon\|x\|_1}.
\end{split}
\end{align}
In particular, for any $\epsilon>0$, with probability one,
\begin{align*}
	\sup\{ d_{\omega_q}([x]_q,[y]_q):y \in \Z^d,\,\|x-y\|_1 \leq \epsilon\|x\|_1 \}<3\Cr{AP}^{(q)}\epsilon\|x\|_1,
\end{align*}
provided $\|x\|_1$ is large enough.
\end{lem}
\begin{proof}
Let $x,y \in \Z^d$ with $\|x-y\|_1 \leq \epsilon\|x\|_1$.
The union bound proves that
\begin{align}\label{eq:d_max}
\begin{split}
	&\P(d_{\omega_q}([x]_q,[y]_q) \geq 3\Cr{AP}\epsilon\|x\|_1)\\
	&\leq 2\P(\mathcal{C}_\infty(\omega_q) \cap B_1(0,\epsilon\|x\|_1)=\emptyset)\\
	&\quad +\sum_{\substack{z \in B_1(x,\epsilon\|x\|_1)\\ w \in B_1(y,\epsilon\|x\|_1)}}
		\P(3\Cr{AP}\epsilon\|x\|_1 \leq d_{\omega_q}(z,w)<\infty).
\end{split}
\end{align}
Since $\Cr{AP}\|z-w\|_1 \leq 3\Cr{AP}\epsilon\|x\|_1$ for all $z \in B_1(x,\epsilon\|x\|_1)$ and $w \in B_1(y,\epsilon\|x\|_1)$, Proposition~\ref{prop:GM} implies that there exists a constant $c$ such that the right side in \eqref{eq:d_max} is bounded from above by
\begin{align*}
	2\Cr{GM1}e^{-\Cr{GM2}\epsilon\|x\|_1}+c(\epsilon\|x\|_1)^{2d}e^{-3\Cr{AP}\Cr{GM2}\epsilon\|x\|_1}.
\end{align*}
Hence, \eqref{eq:maximal} immediately follows.
The second assertion is a direct consequence of the Borel--Cantelli lemma.
\end{proof}

We are now in a position to prove Proposition~\ref{prop:modify}.

\begin{proof}[\bf Proof of Proposition~\ref{prop:modify}]
Proposition~\ref{prop:K_lyap} and the fact that $\mathcal{C}_\infty(\omega_q) \subset \mathcal{C}_\infty(\omega_p)$ imply that on the event $\{ 0 \in \mathcal{C}_\infty(\omega_q) \}$ of positive probability,
\begin{align*}
	\alpha_\lambda^p(x)
	= \lim_{\substack{k \to \infty\\ kx \in \mathcal{C}_\infty(\omega_q)}} \frac{1}{k} a_\lambda(0,kx,\omega_p)
	= \lim_{\substack{k \to \infty\\ kx \in \mathcal{C}_\infty(\omega_q)}} \frac{1}{k} a_\lambda^q(0,kx,\omega_p).
\end{align*}
Thanks to \eqref{eq:chemi} and Lemma~\ref{lem:ible}, $a_\lambda^q(0,kx,\omega_p)$ is integrable and the first assertion follows from the subadditive ergodic theorem for the process $a_\lambda^q(ix,jx,\omega_p)$, $0 \leq i<j,\,i,j \in \N_0$.

For the second assertion, it suffices to show that for any $0<\epsilon \in \Q$, the following holds almost surely:
There exists $N \in \N$ such that for all $x \in \Z^d$ with $\|x\|_1 \geq N$,
\begin{align*}
	|a_\lambda^q(0,x,\omega_p)-\alpha_\lambda^p(x)| \leq \epsilon\|x\|_1.
\end{align*}
To do this, assume that the above statement is false.
Then, there exists $\epsilon_0>0$ such that with positive probability, we can take a sequence $(x_i)_{i=1}^\infty$ of $\Z^d$ satisfying that $\|x_i\|_1 \to \infty$ as $i \to \infty$ and 
\begin{align*}
	|a_\lambda^q(0,x_i,\omega_p)-\alpha_\lambda^p(x_i)|>\epsilon_0\|x_i\|_1,\qquad i \geq 1.
\end{align*}
Without loss of generality, we can assume $x_i/\|x_i\|_1 \to v$ as $i \to \infty$ for some $v \in \S^{d-1}$.
Let $\eta$ be a positive number to be chosen later.
Take $v' \in \Q^d \cap \S^{d-1}$ and $M \in \N$ with $\|v-v'\|_1<\eta$ and  $Mv' \in \Z^d$.
Furthermore, define for $i \geq 1$,
\begin{align*}
	x'_i:=\biggl\lfloor \frac{\|x_i\|_1}{M} \biggr\rfloor Mv'.
\end{align*}
Then, for all large $i$,
\begin{align*}
	\|x_i-x'_i\|_1<\eta\|x_i\|_1+M \leq 2\eta\|x_i\|_1.
\end{align*}
Hence, due to \eqref{eq:triangle} and \eqref{eq:lyap_bd},
\begin{align*}
	\epsilon_0\|x_i\|_1
	&<|a_\lambda^q(0,x_i,\omega_p)-\alpha_\lambda^p(x_i)|\\
	&\leq (\lambda+\log(2d))d_{\omega_q}([x_i]_q,[x'_i]_q)
		+|a_\lambda^q(0,x'_i,\omega_p)-\alpha_\lambda^p(x'_i)|\\
	&\quad +\alpha_\lambda(\xi_1)\|x_i-x'_i\|_1.
\end{align*}
It follows from Lemma~\ref{lem:maximal} and the first assertion that there exists a constant $c$ such that the most right side is not larger than $c\eta\|x_i\|_1$ for all large $i$.
Taking $\eta \leq \epsilon_0/c$, we derives a contradiction and complete the proof.
\end{proof}

\section{Continuity for the Lyapunov exponent}\label{sect:lyap_conti}
The aim of this section is to show Theorem~\ref{thm:lyap_conti}.
To do this, we use the following theorem, which states the upper and the lower semi-continuities of the Lyapunov exponent $\alpha_\lambda^p(x)$ in $p$.

\begin{thm}\label{thm:lyap_ul}
Let $\lambda>0$ and $x \in \Z^d \setminus \{0\}$.
If $p$ and $p_n$, $n \geq 1$, belong to $(p_c,1]$ and $p_n \to p$ as $n \to \infty$, then
\begin{align}\label{eq:lyap_u}
	\limsup_{n \to \infty} \alpha_\lambda^{p_n}(x) \leq \alpha_\lambda^p(x)
\end{align}
and
\begin{align}\label{eq:lyap_l}
	\liminf_{n \to \infty} \alpha_\lambda^{p_n}(x) \geq \alpha_\lambda^p(x).
\end{align}
\end{thm}

We first complete the proof of Theorem~\ref{thm:lyap_conti}, and postpone those of \eqref{eq:lyap_u} and \eqref{eq:lyap_l} until Subsections~\ref{subsect:lyap_u} and \ref{subsect:lyap_l}, respectively.

\begin{proof}[\bf Proof of Theorem~\ref{thm:lyap_conti}]
For $x=0$, \eqref{eq:lyap_conti} is trivial because of $\alpha_\lambda^{p_n}(0)=\alpha_\lambda^p(0)=0$.
Therefore, in the case $x \in \Z^d$, \eqref{eq:lyap_conti} is a direct consequence of Theorem~\ref{thm:lyap_ul}.
For any $x \in \Q^d$, there exists an $M \in \N$ such that $Mx \in \Z^d$, and we have
\begin{align*}
	\lim_{n \to \infty} \alpha_\lambda^{p_n}(x)=\frac{1}{M}\lim_{n \to \infty} \alpha_\lambda^{p_n}(Mx)
	=\frac{1}{M}\alpha_\lambda^p(Mx)=\alpha_\lambda^p(x).
\end{align*}
This means that \eqref{eq:lyap_conti} holds for $x \in \Q^d$.
Let us finally extend it to the case $x \in \R^d$.
Without loss of generality, we may assume that there exists $q \in (p_c,1)$ such that $q<p_n \wedge p$ for all $n \geq 1$.
Let $(x_i)_{i=1}^\infty$ be a sequence of $\Q^d$ with $x_i \to x$ as $i \to \infty$.
Lemma~\ref{lem:ible} tells us that for all $p'>q$,
\begin{align*}
	|\alpha_\lambda^{p'}(x_i)-\alpha_\lambda^{p'}(x)|
	\leq \alpha_\lambda^{p'}(x_i-x)
	\leq \Cr{ible}^{(q)}(\lambda+\log(2d)) \|x_i-x\|_1.
\end{align*}
Therefore,
\begin{align*}
	\lim_{i \to \infty} \sup_{p'>q}|\alpha_\lambda^{p'}(x_i)-\alpha_\lambda^{p'}(x)|=0.
\end{align*}
Note that
\begin{align*}
	&\limsup_{n \to \infty}|\alpha_\lambda^{p_n}(x)-\alpha_\lambda^p(x)|\\
	&\leq 2\sup_{p'>q} |\alpha_\lambda^{p'}(x)-\alpha_\lambda^{p'}(x_i)|
		+\limsup_{n \to \infty} |\alpha_\lambda^{p_n}(x_i)-\alpha_\lambda^p(x_i)|\\
	&= 2\sup_{p'>q} |\alpha_\lambda^{p'}(x)-\alpha_\lambda^{p'}(x_i)|,
\end{align*}
and letting $i \to \infty$ proves \eqref{eq:lyap_conti} for $x \in \R^d$.
\end{proof}

\subsection{Upper semi-continuity for the Lyapunov exponent}\label{subsect:lyap_u}
In this subsection, we prove \eqref{eq:lyap_u} of Theorem~\ref{thm:lyap_ul}.
To this end, let us prepare some notation and lemma.
Let $p_c<q \leq p \leq 1$ and $\lambda>0$, and set $\rho:=4\Cr{AP}^{(q)}(\lambda+\log(2d))/\lambda$.
We define for $x \in \Z^d$,
\begin{align*}
	&\tilde{a}_\lambda^q(0,x,\omega_p)\\
	&:= -\log E_{\omega_p}^{[0]_q}\bigl[ e^{-\lambda H([x]_q)}\1{\{ H([x]_q) \leq \rho\|x\|_1 \}} \bigr]
		\wedge (\lambda+\log(2d))d_{\omega_q}([0]_q,[x]_q).
\end{align*}

The following lemma says that the expectations of $a_\lambda^q(0,kx,\omega_p)$ and $\tilde{a}_\lambda^q(0,kx,\omega_p)$ are comparable uniformly in $p>q$.

\begin{lem}\label{lem:a_unf}
For $x \in \Z^d \setminus \{0\}$,
\begin{align}\label{eq:a_unf}
	\lim_{k \to \infty} \frac{1}{k}\sup_{p'>q}
	\E\bigl[ |a_\lambda^q(0,kx,\omega_{p'})-\tilde{a}_\lambda^q(0,kx,\omega_{p'})| \bigr]=0.
\end{align}
\end{lem}
\begin{proof}
We first show that for each $y \in \Z^d \setminus \{0\}$, there exists an event $\Gamma_y(q)$ with $\P(\Gamma_y(q)^c) \leq \Cr{maximal1}^{(q)}e^{-\Cr{maximal2}^{(q)}\|y\|_1}$ such that
\begin{align}\label{eq:a_dif}
	|\tilde{a}_\lambda^q(0,y,\omega_p)-a_\lambda^q(0,y,\omega_p)| \leq \log 2.
\end{align}
Note that
\begin{align*}
	E_{\omega_p}^{[0]_q}\bigl[ e^{-\lambda H([y]_q)}\1{\{ \rho\|y\|_1<H([y]_q)<\infty \}} \bigr]
	\leq e^{-\lambda \rho\|y\|_1}.
\end{align*}
On the event $\Gamma_y(q):=\{ d_{\omega_q}([0]_q,[y]_q) \leq 3\Cr{AP}^{(q)}\|y\|_1 \}$,
\begin{align}\label{eq:a_divide}
	3\Cr{AP}^{(q)}(\lambda+\log(2d))\|y\|_1
	\geq a_\lambda^q(0,y,\omega_p)
	\geq -\log\bigl( e^{-\tilde{a}_\lambda^q(0,y,\omega_p)}+e^{-\lambda \rho\|y\|_1} \bigr).
\end{align}
If $\tilde{a}_\lambda^q(0,y,\omega_p)>\lambda \rho\|y\|_1$, then we have on $\Gamma_y(q)$,
\begin{align*}
	3\Cr{AP}^{(q)}(\lambda+\log(2d))\|y\|_1>-\log 2+\lambda \rho\|y\|_1,
\end{align*}
which yields
\begin{align*}
	\rho<\frac{3\Cr{AP}^{(q)}(\lambda+\log(2d))\|y\|_1+\log 2}{\lambda\|y\|_1}
	\leq \frac{4\Cr{AP}^{(q)}(\lambda+\log(2d))}{\lambda}.
\end{align*}
This contradicts the definition of $\rho$, and hence $\tilde{a}_\lambda^q(0,y,\omega_p) \leq \lambda \rho\|y\|_1$ must hold on $\Gamma_y(q)$.
Therefore, by \eqref{eq:a_divide}, one has on $\Gamma_y(q)$,
\begin{align*}
	a_\lambda^q(0,y,\omega_p) \geq -\log 2+\tilde{a}_\lambda^q(0,y,\omega_p).
\end{align*}
By definition, $\tilde{a}_\lambda^q(0,y,\omega_p)$ is always bigger than or equal to $a_\lambda^q(0,y,\omega_p)$, and \eqref{eq:a_dif} holds on $\Gamma_y(q)$.
We use Lemma~\ref{lem:maximal} to obtain the desired bound for $\P(\Gamma_y(q)^c)$, and the assertion follows.

Let us next prove \eqref{eq:a_unf}.
The above assertion says that for each $p'>q$,
\begin{align*}
	&\E\bigl[ |a_\lambda^q(0,kx,\omega_{p'})-\tilde{a}_\lambda^q(0,kx,\omega_{p'})| \bigr]\\
	&\leq \E\bigl[ (\tilde{a}_\lambda^q(0,kx,\omega_{p'})-a_\lambda^q(0,kx,\omega_{p'}))\1{\Gamma_{kx}(q)} \bigr]
		+\E\bigl[ \tilde{a}_\lambda^q(0,kx,\omega_p)\1{\Gamma_{kx}(q)^c} \bigr]\\
	&\leq \log 2+(\lambda+\log(2d))\E[d_{\omega_q}([0]_q,[kx]_q)\1{\Gamma_{kx}(q)^c}].
\end{align*}
Schwarz's inequality and Lemma~\ref{lem:ible} imply that for some constant $c$ (which is independent of $p'$),
\begin{align*}
	\E[d_{\omega_q}([0]_q,[kx]_q)\1{\Gamma_{kx}(q)^c}]
	&\leq \E[d_{\omega_q}([0]_q,[kx]_q)^2]^{1/2}\P(\Gamma_{kx}(q)^c)^{1/2}\\
	&\leq c\|kx\|_1 e^{-\Cr{maximal2}^{(q)}\|kx\|_1/2}.
\end{align*}
Therefore,
\begin{align*}
	&\limsup_{k \to \infty} \frac{1}{k}\sup_{p'>q} \E\bigl[ |a_\lambda^q(0,kx,\omega_{p'})-\tilde{a}_\lambda^q(0,kx,\omega_{p'})| \bigr]\\
	&\leq \limsup_{k \to \infty} \frac{1}{k} \bigl\{ \log 2
		+c(\lambda+\log(2d))\|kx\|_1 e^{-\Cr{maximal2}^{(q)}\|kx\|_1/2} \bigr\}=0,
\end{align*}
and the proof is complete.
\end{proof}

Now we are in a position to prove \eqref{eq:lyap_u} of Theorem~\ref{thm:lyap_ul}.

\begin{proof}[\bf Proof of (\ref{eq:lyap_u}) in Theorem~\ref{thm:lyap_ul}]
Pick $q \in (p_c,1)$ such that $q<p_n \wedge p$ for all large $n$.
Hence, if $n$ is large enough, then $\alpha_\lambda^{p_n}(x)-\alpha_\lambda^p(x)$ is bounded from above by
\begin{align*}
	\frac{1}{k}\bigl( \E[a_\lambda^q(0,kx,\omega_{p_n})]-\E[a_\lambda^q(0,kx,\omega_p)] \bigr)
	+\biggl( \frac{1}{k} \E[a_\lambda^q(0,kx,\omega_p)]-\alpha_\lambda^p(x) \biggr).
\end{align*}
It is clear from Proposition~\ref{prop:modify} that the second term in the above expression converges to zero as $k \to \infty$.
The task is now to prove
\begin{align}\label{eq:Leb}
	\lim_{k \to \infty}\lim_{n \to \infty}
	\frac{1}{k}\bigl( \E[a_\lambda^q(0,kx,\omega_{p_n})]-\E[a_\lambda^q(0,kx,\omega_p)] \bigr)=0.
\end{align}
To this end,
\begin{multline}\label{eq:upper_sup}
	\bigl| \E[a_\lambda^q(0,kx,\omega_{p_n})]-\E[a_\lambda^q(0,kx,\omega_p)] \bigr|\\
	\leq \bigl| \E[\tilde{a}_\lambda^q(0,kx,\omega_{p_n})]-\E[\tilde{a}_\lambda^q(0,kx,\omega_p)] \bigr|\\
	+2\sup_{p'>q}\E\bigl[ |a_\lambda^q(0,kx,\omega_{p'})-\tilde{a}_\lambda^q(0,kx,\omega_{p'})| \bigr].
\end{multline}
For the first term of the right side in \eqref{eq:upper_sup}, note that $\tilde{a}_\lambda^q(0,kx,\omega_{p_n})$ depends only on $\omega_q$ and $\omega_{p_n}(e)$'s on edges $e$ intersecting $B_1([0]_q,\rho\|kx\|_1)$.
This implies
\begin{align*}
	\lim_{n \to \infty} \tilde{a}_\lambda^q(0,kx,\omega_{p_n})=\tilde{a}_\lambda^q(0,kx,\omega_p).
\end{align*}
Moreover, $\tilde{a}_\lambda^q(0,kx,\omega_{p_n})$ is dominated by $(\lambda+\log(2d))d_{\omega_q}([0]_q,[kx]_q)$ uniformly in $n$, and Lebesgue's dominated convergence theorem gives that for each $k$,
\begin{align*}
	\lim_{n \to \infty} \bigl| \E[\tilde{a}_\lambda^q(0,kx,\omega_{p_n})]-\E[\tilde{a}_\lambda^q(0,kx,\omega_p)] \bigr|=0.
\end{align*}
On the other hand, from Lemma~\ref{lem:a_unf}, the last term of \eqref{eq:upper_sup} divided by $k$ converges to zero as $k \to \infty$, and \eqref{eq:Leb} follows.
\end{proof}

\subsection{Lower semi-continuity for the Lyapunov exponent}\label{subsect:lyap_l}
Our goal in this subsection is to show \eqref{eq:lyap_l} of Theorem~\ref{thm:lyap_ul}.
The following lower large deviation estimate plays the key role of the proof.

\begin{prop}\label{prop:lld}
Let $\lambda,\epsilon>0$.
Assume that $p$ and $p_n$, $n \geq 1$, belong to $(p_c,1]$ and $p_n \to p$ as $n \to \infty$.
Then, there exist constants $\Cl{lld1}$ and $\Cl{lld2}$ such that if $n$ is large enough, then for all large $t$ and $z \in \Z^d \setminus \{0\}$ with $\alpha_\lambda^p(z) \geq t$,
\begin{align*}
	\P\bigl( a_\lambda^{p_n}(0,z,\omega_{p_n})<t(1-2\epsilon) \bigr)
	\leq \Cr{lld1}e^{-\Cr{lld2}t}.
\end{align*}
\end{prop}

Before proving this proposition, we complete the proof of \eqref{eq:lyap_l} in Theorem~\ref{thm:lyap_ul}.

\begin{proof}[\bf Proof of \eqref{eq:lyap_l} in Theorem~\ref{thm:lyap_ul}]
Given $\epsilon>0$ and $x \in \Z^d \setminus \{0\}$,  we use Proposition~\ref{prop:lld} and the Borel--Cantelli lemma to obtain that for all large $n$, with probability one,
\begin{align*}
	\alpha_\lambda^{p_n}(x)
	= \liminf_{k \to \infty} \frac{1}{k} a_\lambda^{p_n}(0,kx,\omega_{p_n})
	\geq \alpha_\lambda^p(x)(1-2\epsilon).
\end{align*}
Therefore, \eqref{eq:lyap_l} follows by letting $n \to \infty$ and $\epsilon \searrow 0$.
\end{proof}

It remains to prove Proposition~\ref{prop:lld}.
We follow the approach taken in \cite[Section~3]{Ahl15}.
First of all, let us choose appropriate constants for our proof.
Given $\lambda,\epsilon>0$ and $p \in (p_c,1]$, fix $r \in \Q$ and $\delta \in (1/2,1)$ with
\begin{align*}
	0<r<\frac{\epsilon}{6d\Cr{AP}^{(p)}(\lambda+\log(2d))\alpha_\lambda^p(\xi_1)}
\end{align*}
and
\begin{align*}
	\delta^3 \geq \frac{1-2\epsilon}{1-\epsilon}(1+2dr\alpha_\lambda^p(\xi_1)).
\end{align*}
In addition, pick $\kappa \in (\delta,1)$ such that
\begin{align*}
	D(\delta \| \kappa):=\delta \log\frac{\delta}{\kappa}+(1-\delta)\log\frac{1-\delta}{1-\kappa}>\log \frac{2}{r}.
\end{align*}

We next prepare some notation and lemmata.
Set for $t \geq 0$,
\begin{align*}
	\mathcal{B}(t):=\{ y \in \R^d:\alpha_\lambda^p(y) \leq t \}.
\end{align*}
It is clear from Proposition~\ref{prop:K_lyap} that $\mathcal{B}(t)$ is a nonrandom, compact, convex set of $\R^d$ with $B_1(0,t/\alpha_\lambda^p(\xi_1)) \subset \mathcal{B}(t) \subset B_1(0,t)$.
By the choice of $r$, one has $[-2r,2r)^d \subset \mathcal{B}(1)$.
Then, define for $v \in \Z^d$, $\ell \in \N$ and $p' \in (p_c,1]$,
\begin{align*}
	c(v,\ell,\omega_{p'})
	:= \inf\Bigl\{ -\log E_{\omega_{p'}}^x\bigl[ e^{-\lambda T(v,\ell)} \bigr]: x \in 2r\ell v+[-r\ell,r\ell)^d \Bigr\},
\end{align*}
where
\begin{align*}
	T(v,\ell):=\inf\{ n \geq 0:X_n \not\in 2r\ell v+\mathcal{B}(\ell) \}.
\end{align*}

\begin{lem}\label{lem:c}
We have almost surely,
\begin{align*}
	\liminf_{\ell \to \infty} \frac{1}{\ell}c(0,\ell,\omega_p)>1-\epsilon.
\end{align*}
\end{lem}
\begin{proof}
Proposition~\ref{prop:GM}-(ii), (iii) and the Borel--Cantelli lemma give that with probability one, for all large $\ell$, $\mathcal{C}_\infty(\omega_p) \cap [-r\ell,r\ell)^d \not= \emptyset$ and $\#\mathcal{C}_x(\omega_p)=\infty$ for all $x \in [-r\ell,r\ell)^d$ with $\# \mathcal{C}_x(\omega_p) \geq r\ell$.
Therefore, we can restrict ourselves to this event.

Since $[0]_p \in [-r\ell,r\ell)^d$, $c(0,\ell,\omega_p)$ is finite.
This means that there exist $x_\ell,y_\ell \in \mathcal{O}(\omega_p)$ such that $x_\ell$ and $y_\ell$ are linked by an open lattice path,
\begin{align*}
	x_\ell \in [-r\ell,r\ell)^d,\qquad \ell<\alpha_\lambda^p(y_\ell) \leq \ell+\alpha_\lambda^p(\xi_1)
\end{align*}
and
\begin{align*}
	c(0,\ell,\omega_p)
	\geq a_\lambda(x_\ell,y_\ell,\omega_p)-\log\#\bar{\partial}\mathcal{B}(\ell),
\end{align*}
where $\bar{\partial}\mathcal{B}(\ell)$ denotes the outer boundary of $\mathcal{B}(\ell)$ on $\Z^d$, i.e.,
\begin{align*}
	\bar{\partial}\mathcal{B}(\ell)=\{ y \in \Z^d \setminus \mathcal{B}(\ell):
	\exists z \in \mathcal{B}(\ell) \cap \Z^d \text{ such that } \|y-z\|_1=1 \}.
\end{align*}
In particular, $\#\mathcal{C}_{x_\ell}(\omega_p) \geq r\ell$ holds by the choice of $r$.
Hence, both $x_\ell$ and $y_\ell$ are included in $\mathcal{C}_\infty(\omega_p)$, and
\begin{align*}
	|a_\lambda(x_\ell,y_\ell,\omega_p)-a_\lambda^p(0,y_\ell,\omega_p)|
	\leq (\lambda+\log(2d))d_{\omega_p}([0]_p,x_\ell).
\end{align*}
We use Lemma~\ref{lem:maximal} to obtain that for all large $\ell$, the right side is not larger than
\begin{align*}
	3\Cr{AP}^{(p)}(\lambda+\log(2d))\|x_\ell\|_1< \frac{\epsilon\ell}{2}.
\end{align*}
Since $\|y_\ell\|_1 \to \infty$ as $\ell \to \infty$, this together with Proposition~\ref{prop:modify} proves
\begin{align*}
	\liminf_{\ell \to \infty} \frac{1}{\ell}c(0,\ell,\omega_p)
	&\geq \liminf_{\ell \to \infty} \frac{1}{\ell} a_\lambda^p(0,y_\ell,\omega_p)
		-\limsup_{\ell \to \infty} \frac{1}{\ell}\log\#\bar{\partial}\mathcal{B}(\ell)-\frac{\epsilon}{2}\\
	&> 1-\epsilon,
\end{align*}
and the lemma follows.
\end{proof}

\begin{lem}\label{lem:c_n}
Assume that a sequence $(p_n)_{n=1}^\infty$ of $(p_c,1]$ converges to $p$ as $n \to \infty$.
Then, for each $\ell \in \N$,
\begin{align*}
	\lim_{n \to \infty} \P\bigl( c(0,\ell,\omega_{p_n})>\ell(1-\epsilon) \bigr)
	= \P\bigl( c(0,\ell,\omega_p)>\ell(1-\epsilon) \bigr).
\end{align*}
\end{lem}
\begin{proof}
Since $\mathcal{B}(\ell)$ is bounded, with probability one, if $n$ is large enough, then $\omega_{p_n}(e)=\omega_p(e)$ holds for all $e \in \mathcal{E}^d$ intersecting $\mathcal{B}(\ell)$.
Since $c(0,\ell,\cdot)$ depends only on the configurations of edges intersecting $\mathcal{B}(\ell)$, with probability one, $c(0,\ell,\omega_{p_n})=c(0,\ell,\omega_p)$ holds for all large $n$.
This yields that
\begin{align*}
	\lim_{n \to \infty}\1{\{ c(0,\ell,\omega_{p_n})>\ell(1-\epsilon) \}}
	= \1{\{ c(0,\ell,\omega_p)>\ell(1-\epsilon) \}}.
\end{align*}
Accordingly, the lemma immediately follows from Lebesgue's dominated convergence theorem.
\end{proof}

We say that a site $v \in \Z^d$ is \emph{$(\ell,p_n)$-good} if
\begin{align*}
	c(v,\ell,\omega_{p_n})>\ell(1-\epsilon).
\end{align*}
Note that $(\1{\{ v \text{ is $(\ell,p_n)$-good} \}})_{v \in \Z^d}$ is a finitely dependent family of random variables taking values in $\{ 0,1 \}$.
In addition, from Lemmata~\ref{lem:c} and \ref{lem:c_n},
\begin{align*}
	\lim_{\ell \to \infty}\lim_{n \to \infty} \inf_{v \in \Z^d} \P(v \text{ is $(\ell,p_n)$-good})=1.
\end{align*}
It follows from Proposition~\ref{prop:domi} that there exist $\ell,N \in \N$ such that for all $n \geq N$, $(\1{\{ v \text{ is $(\ell,p_n)$-good} \}})_{v \in \Z^d}$ stochastically dominates the independent Bernoulli site percolation $\eta_\kappa=(\eta_\kappa(v))_{v \in \Z^d}$ of parameter $\kappa$.

From now on, fix $\ell$, $N$ and $q \in (p_c,1)$ such that $r\ell \in \N$, $q<p_n \wedge p$ for all $n \geq N$ and the stochastic domination above is established.\qed

\begin{proof}[\bf Proof of Proposition~\ref{prop:lld}]
For each $m \geq 2$, denote by $\mathcal{W}_m$ the set of all sequences $w=(w_1,\dots,w_m)$ of distinct points of $\Z^d$ such that $\|w_1\|_1 \leq 2t$ and $\|w_i-w_{i+1}\|_1 \leq 2/r$ for $1 \leq i \leq m-1$.
Then, our first claim is that there exist constants $c$ and $c'$ (which are independent of $p_n$'s) such that for all $n \geq N$,
\begin{align}\label{eq:SD}
\begin{split}
	&\P\Biggl( \exists m \geq t/(8\ell),\,\exists w \in \mathcal{W}_m \text{ such that }
		\sum_{i=1}^m\1{\{ w_i \text{ is $(\ell,p_n)$-good} \}}<\delta m \Biggr)\\
	&\leq ce^{-c't}.
\end{split}
\end{align}
From the union bound and the stochastic domination, the left side of \eqref{eq:SD} is not larger than
\begin{align*}
	\sum_{w \in \mathcal{W}_m} \P\Biggl( \sum_{i=1}^m\1{\{ w_i \text{ is $(\ell,p_n)$-good} \}}<\delta m \Biggr)
	\leq \sum_{w \in \mathcal{W}_m} \P\Biggl( \sum_{i=1}^m\eta_\kappa(w_i)<\delta m \Biggr).
\end{align*}
A standard calculation shows that all the probabilities in the right side are smaller than or equal to $e^{-mD(\delta\| \kappa)}$ and the total number of choices for $w \in \mathcal{W}_m$ is of order $t^d(2/r)^{m-1}$.
Thus, since $m \geq t/(8\ell)$, \eqref{eq:SD} immediately follows from the choice of $\kappa$.

We move to the proof of Proposition~\ref{prop:lld}.
Fix $n \geq N$, $t \geq 4\ell/(1-\delta)$ and $z \in \Z^d \setminus \{0\}$ with $\alpha_\lambda^p(z) \geq t$.
The boxes $2r\ell v+[-r\ell,r\ell)^d$, $v \in \Z^d$, form a partition of $\Z^d$, and each $x \in \Z^d$ is contained in precisely one box.
Thus, write $x^*$ for the index such that $x \in 2r\ell x^*+[-r\ell,r\ell)^d$.
We now introduce the stopping times $(\tau_i)_{i=-1}^\infty$ of the filtration $\mathcal{F}_k:=\sigma(X_0,\dots,X_k)$, $k \geq 0$, as follows:
\begin{align*}
	&\tau_{-1}=\tau_0:=0,\\
	&\tau_{i+1}:=\inf\{ k>\tau_i: X_k \not\in 2r\ell X_{\tau_i}^*+\mathcal{B}(\ell) \},\qquad i \geq 0.
\end{align*}
In addition, define $\rho_0:=-1$, and by induction for $i \geq 0$,
\begin{align*}
	\rho_{i+1}:=\inf\{ j>\rho_i: X_{\tau_j}^* \text{ is $(\ell,p_n)$-good} \}.
\end{align*}
Then, $\tau_{\rho_i}$ is a stopping time of the filtration $(\mathcal{F}_k)_{k=0}^\infty$ and $X_{\tau_{\rho_i}}^*$ is $(\ell,p_n)$-good.

Consider the event $\Gamma_t$ that for all $m \geq t/(8\ell)$ and $w \in \mathcal{W}_m$,
\begin{align*}
	\sum_{i=1}^m\1{\{ w_i \text{ is $(\ell,p_n)$-good} \}} \geq \delta m,
\end{align*}
and $\mathcal{C}_\infty(\omega_q)$ intersects both $B_1(0,(1-\delta)t/(2\alpha_\lambda^p(\xi_1)))$ and $B_1(z,(1-\delta)t/(2\alpha_\lambda^p(\xi_1)))$.
Denote by $M$ the number of distinct ($2r\ell v+\mathcal{B}(\ell)$)'s from which the random walk starting at $[0]_{p_n}$ exits before reaching $[z]_{p_n}$.
We then have on $\Gamma_t$,
\begin{align*}
	t &\leq \alpha_\lambda^p(z-[z]_{p_n})+\alpha_\lambda^p([z]_{p_n}-[0]_{p_n})+\alpha_\lambda^p([0]_{p_n})\\
	&\leq (1-\delta)t+(M+1)\bigl\{ (1+dr\alpha_\lambda^p(\xi_1))\ell+\alpha_\lambda^p(\xi_1) \bigr\}.
\end{align*}
By the choice of $r$, $\ell$ and $\delta$,
\begin{align*}
	M \geq \frac{\delta t}{(1+2dr\alpha_\lambda^p(\xi_1))\ell}-1
\end{align*}
and
\begin{align*}
	\frac{(\delta-\delta^2)t}{(1+2dr\alpha_\lambda^p(\xi_1))\ell} \geq \frac{(1-\delta)t}{4\ell} \geq 1.
\end{align*}
Therefore,
\begin{align*}
	M \geq \frac{\delta^2 t}{(1+2dr\alpha_\lambda^p(\xi_1))\ell}
	\geq \frac{t}{8\ell}.
\end{align*}
To shorten notation,
\begin{align*}
	\nu:=\biggl\lceil \frac{\delta^3t}{(1+2dr\alpha_\lambda^p(\xi_1))\ell} \biggr\rceil.
\end{align*}
If $t$ is large enough, then on the event $\Gamma_t$, one has $\tau_{\rho_{\nu}}<H([z]_{p_n})$ $P_{\omega_{p_n}}^{[0]_{p_n}} \hyphen \as$, and
\begin{align*}
	&a_\lambda^{p_n}(0,z,\omega_{p_n})\\
	&\geq -\log E_{\omega_{p_n}}^{[0]_{p_n}} \Biggl[
		\exp\bigl\{ -\lambda \bigl( \tau_{\rho_{\nu}+1}-\tau_{\rho_{\nu}} \bigr) \bigr\}
		\prod_{i=0}^{\nu-1} e^{-\lambda(\tau_{\rho_{i+1}}-\tau_{\rho_i})} \Biggr].
\end{align*}
We use the strong Markov property with respect to $\tau_{\rho_i}$, $1 \leq i \leq \nu$, and the fact that $X_{\tau_{\rho_i}}^*$'s are $(\ell,p_n)$-good to obtain
\begin{align*}
	E_{\omega_{p_n}}^{[0]_{p_n}} \Biggl[ \prod_{i=0}^{\nu-1} e^{-\lambda(\tau_{\rho_{i+1}}-\tau_{\rho_i})} \Biggr]
	\leq e^{-\nu \ell(1-\epsilon)}.
\end{align*}
The choice of $\delta$ guarantees that on $\Gamma_t$,
\begin{align*}
	a_\lambda^{p_n}(0,z,\omega_{p_n}) \geq \nu\ell(1-\epsilon)>t(1-2\epsilon).
\end{align*}
Hence,
\begin{align*}
	\P\bigl( a_\lambda^{p_n}(0,z,\omega_{p_n})<t(1-2\epsilon) \bigr) \leq \P(\Gamma_t^c),
\end{align*}
and the proposition follows from Proposition~\ref{prop:GM} and \eqref{eq:SD}.
\end{proof}

\section{Continuity for the rate function}\label{sect:rate_conti}
The aim of this section is to prove Theorem~\ref{thm:rate_conti}.
To begin with, in Subsection~\ref{subsect:relation}, we observe a relation between the Lyapunov exponent and the so-called time constant for the chemical distance.
The proof of Theorem~\ref{thm:rate_conti} is given in Subsection~\ref{subsect:rate_conti}.

\subsection{Relation to the time constant for the chemical distance}\label{subsect:relation}
Garet and Marchand have investigated the asymptotic behavior of the chemical distance.
The following proposition is one of their results~\cite{GarMar04}, which states that the chemical distance is asymptotically equivalent to a deterministic
norm on $\R^d$.

\begin{prop}\label{prop:tconst}
Let $p \in (p_c,1]$.
There exists a norm $\mu^p(\cdot)$ on $\R^d$ (which is called the time constant) such that almost surely on the event $\{ 0 \in \mathcal{C}_\infty(\omega_p) \}$,
\begin{align*}
	\lim_{\substack{k \to \infty\\ kx \in \mathcal{C}_\infty(\omega_p)}}
	\frac{1}{k}d_{\omega_p}(0,kx)=\mu^p(x).
\end{align*}
\end{prop}

The next proposition is our objective of this subsection.
It says that the travel cost and the Lyapunov exponent converge decreasingly to the chemical distance and the time constant, respectively.
As stated in Corollary~\ref{cor:shape} below, this is useful to determine the effective domain of the rate function.

\begin{prop}\label{prop:relation}
Let $p_c<q \leq p \leq 1$.
For all $x,y \in \Z^d$, $\P \hyphen \as$,
\begin{align*}
	\frac{a_\lambda^q(x,y,\omega_p)}{\lambda} \searrow d_{\omega_p}([x]_q,[y]_q) \qquad \text{as} \quad \lambda \to \infty.
\end{align*}
In addition, for all $x \in \R^d$,
\begin{align*}
	\frac{\alpha_\lambda^p(x)}{\lambda} \searrow \mu^p(x) \qquad \text{as} \quad \lambda \to \infty.
\end{align*}
\end{prop}
\begin{proof}
We follow the strategy taken in \cite[Proposition~9]{Zer98a}.
First observe that for any $\lambda>0$,
\begin{align*}
	\frac{a_\lambda^q(x,y,\omega_p)}{\lambda}
	\geq -\frac{1}{\lambda} \log E_{\omega_p}^{[x]_q}\bigl[ e^{-\lambda d_{\omega_p}([x]_q,[y]_q)} \1{\{ H([y]_q)<\infty \}} \bigr]
	\geq d_{\omega_p}([x]_q,[y]_q).
\end{align*}
If $0<\lambda_1<\lambda_2$, then by Jensen's inequality,
\begin{align}\label{eq:lamda}
	\frac{a_{\lambda_2}^q(x,y,\omega_p)}{\lambda_2}
	\leq \frac{a_{\lambda_1}^q(x,y,\omega_p)}{\lambda_1}.
\end{align}
Furthermore, by \eqref{eq:chemi},
\begin{align*}
	\limsup_{\lambda \to \infty} \frac{a_\lambda^q(x,y,\omega_p)}{\lambda}
	\leq \limsup_{\lambda \to \infty} \biggl( 1+\frac{\log(2d)}{\lambda} \biggr) d_{\omega_p}([x]_q,[y]_q)
	= d_{\omega_p}([x]_q,[y]_q),
\end{align*}
and the first assertion follows.

For the second assertion, we may assume $x \in \Z^d$.
It follows from \eqref{eq:lamda} and the first assertion that $\alpha_\lambda^p(x)/\lambda$ decreases as $\lambda \to \infty$ to
\begin{align*}
	\inf_{\lambda \in \N} \frac{\alpha_\lambda^p(x)}{\lambda}
	= \inf_{k \in \N} \frac{1}{k} \inf_{\lambda \in \N} \E\biggl[ \frac{a_\lambda^q(0,kx,\omega_p)}{\lambda} \biggr]
	= \inf_{k \in \N} \frac{1}{k} \E[d_{\omega_p}([0]_q,[kx]_q)].
\end{align*}
The most right side is equal to $\mu^p(x)$ by using Proposition~\ref{prop:tconst} and the same strategy as in Proposition~\ref{prop:modify}, and the proof is complete.
\end{proof}

The following corollary is an immediate consequence of Proposition~\ref{prop:relation}.

\begin{cor}\label{cor:shape}
Let $p \in (p_c,1]$.
Then, we have $\mathcal{D}_{I^p}=\{ x \in \R^d:\mu^p(x) \leq 1 \}$.
In particular, $I^p(x) \leq \log(2d)$ holds for $x \in \mathcal{D}_{I^p}$.
\end{cor}
\begin{proof}
Proposition~\ref{prop:tconst}, together with \eqref{eq:chemi} and Proposition~\ref{prop:K_lyap}, proves that for any $\lambda \geq 0$ and $x \in \R^d$,
\begin{align*}
	\alpha_\lambda^p(x) \leq (\lambda+\log(2d))\mu^p(x).
\end{align*}
This enables us to show that for each $\lambda \geq 0$ and $x \in \R^d$ with $\mu^p(x) \leq 1$,
\begin{align*}
	\alpha_\lambda^p(x)-\lambda
	\leq (\lambda+\log(2d))\mu^p(x)-\lambda
	\leq \log(2d).
\end{align*}
Hence, for all $x \in \R^d$ with $\mu^p(x) \leq 1$,
\begin{align*}
	I^p(x)=\sup_{\lambda \geq 0}(\alpha_\lambda^p(x)-\lambda) \leq \log(2d),
\end{align*}
and $\mathcal{D}_{I^p} \supset \{ x \in \R^d:\mu^p(x) \leq 1 \}$ holds.

For the converse inclusion, assume that $x \in \R^d$ satisfies $\mu^p(x)>1$.
Then, Proposition~\ref{prop:relation} implies
\begin{align*}
	I^p(x) \geq \sup_{\lambda>0} \lambda\biggl( \frac{\alpha_\lambda^p(x)}{\lambda}-1 \biggr)
	\geq \sup_{\lambda>0} \lambda(\mu^p(x)-1)=\infty.
\end{align*}
This leads to $\mathcal{D}_{I^p} \subset \{ x \in \R^d:\mu^p(x) \leq 1 \}$, and we complete the proof.
\end{proof}

\subsection{Proof of Theorem~\ref{thm:rate_conti}}\label{subsect:rate_conti}
To prove Theorem~\ref{thm:rate_conti}, let us introduce for each $p \in (p_c,1]$ and $x \in \R^d$,
\begin{align*}
	\lambda_+^p(x):=\sup\{ \lambda>0:\partial_-\alpha_\lambda^p(x) \geq 1 \}
\end{align*}
and
\begin{align*}
	\lambda_-^p(x):=\inf\{ \lambda>0:\partial_-\alpha_\lambda^p(x) \leq 1 \},
\end{align*}
where $\partial_-\alpha_\lambda^p(x)$ is the left-derivative of $\alpha_\lambda^p(x)$ in $\lambda$.
Roughly speaking, the slope of $\alpha_\lambda^p(x)$ in $\lambda$ is equal to one between $\lambda_-^p(x)$ and $\lambda_+^p(x)$, and is strictly larger (resp.~smaller) than one for $\lambda<\lambda_-^p(x)$ (resp.~$\lambda>\lambda_+^p(x)$).
This means that both $\lambda_+^p(x)$ and $\lambda_-^p(x)$ attain the supremum in \eqref{eq:rate_def}, which is the definition of the rate function.

\begin{lem}\label{lem:attain}
We have $\lambda_-^p(x) \leq \lambda_+^p(x)$.
Moreover, if $\lambda_+^p(x)<\infty$, then
\begin{align*}
	I^p(x)=\alpha_{\lambda_+^p(x)}^p(x)-\lambda_+^p(x)=\alpha_{\lambda_-^p(x)}^p(x)-\lambda_-^p(x).
\end{align*}
\end{lem}
\begin{proof}
The first assertion is trivial.
This is because
\begin{align*}
	\lambda_-^p(x)
	&\leq \inf\{ \lambda>0:\partial_-\alpha_\lambda^p(x)=1 \}\\
	&\leq \sup\{ \lambda>0:\partial_-\alpha_\lambda^p(x) \geq 1 \}=\lambda_+^p(x).
\end{align*}

For the second assertion, we assume $\lambda_+^p(x)<\infty$.
It follows from the definition of $\lambda_+^p(x)$ that $\partial_-\alpha_{\lambda_0}^p(x)<1$ for any $\lambda_0>\lambda_+^p(x)$.
We use the concavity of $\alpha_\lambda^p(x)$ in $\lambda$ to obtain that for all $\lambda>\lambda_0$,
\begin{align*}
	\frac{\alpha_\lambda^p(x)-\alpha_{\lambda_0}^p(x)}{\lambda-\lambda_0}
	\leq \partial_- \alpha_{\lambda_0}^p(x)<1,
\end{align*}
which proves $\alpha_\lambda^p(x)-\lambda>\alpha_{\lambda_0}^p(x)-\lambda_0$ for $\lambda>\lambda_0>\lambda_+^p(x)$.
On the other hand, for any $\epsilon>0$, one can find $\lambda_\epsilon \in (\lambda_+^p(x)-\epsilon,\lambda_+^p(x)]$ with $\partial_- \alpha_{\lambda_\epsilon}^p(x) \geq 1$.
The concavity also implies that for all $\lambda \in [0,\lambda_\epsilon)$,
\begin{align*}
	\frac{\alpha_{\lambda_\epsilon}^p(x)-\alpha_\lambda^p(x)}{\lambda_\epsilon-\lambda}
	\geq \partial_- \alpha_{\lambda_\epsilon}^p(x) \geq 1,
\end{align*}
and $\alpha_{\lambda_\epsilon}^p(x)-\lambda_\epsilon \geq \alpha_\lambda^p(x)-\lambda$ holds for all $\lambda \in [0,\lambda_\epsilon)$.
Consequently,
\begin{align*}
	I^p(x)
	&\leq \sup_{\lambda_\epsilon \leq \lambda \leq \lambda_0}(\alpha_\lambda^p(x)-\lambda)\\
	&\leq (\alpha_{\lambda_+^p(x)}^p(x)-\lambda_+^p(x))
		+\sup_{\lambda_\epsilon \leq \lambda \leq \lambda_0}
		\bigl\{ (\alpha_\lambda^p(x)-\lambda)-(\alpha_{\lambda_+^p(x)}^p(x)-\lambda_+^p(x)) \bigr\}.
\end{align*}
Note that the function $\lambda \longmapsto \alpha_\lambda^p(x)-\lambda$ is continuous on $[0,\infty)$.
Hence, since $\alpha_{\lambda_+^p(x)}^p(x)-\lambda_+^p(x) \leq I^p(x)$, we have $I^p(x)=\alpha_{\lambda_+^p(x)}^p(x)-\lambda_+^p(x)$ by letting $\epsilon \searrow 0$ and $\lambda_0 \searrow \lambda_+^p(x)$.
A similar argument is applicable for $\lambda_-^p(x)$, and the proof is complete.
\end{proof}

The following lemma gives the lower and the upper semi-continuities of $\lambda_-^p(x)$ and $\lambda_+^p(x)$ in $p$, respectively.
This plays an important role in the proof of Theorem~\ref{thm:rate_conti}.

\begin{lem}\label{lem:lambda}
Suppose that $p$ and $p_n$, $n \geq 1$, belong to $(p_c,1]$.
If $\lambda_+^p(x)<\infty$ and $p_n \to p$ as $n \to \infty$, then
\begin{align}\label{eq:limL1}
	\liminf_{n \to \infty} \lambda_-^{p_n}(x) \geq \lambda_-^p(x)
\end{align}
and
\begin{align}\label{eq:limL2}
	\limsup_{n \to \infty} \lambda_+^{p_n}(x) \leq \lambda_+^p(x).
\end{align}
\end{lem}
\begin{proof}
We first observe \eqref{eq:limL1}.
This is trivial in the case $\lambda_-^p(x)=0$.
Hence, we may assume $\lambda_-(x)>0$.
Note that $\partial_- \alpha_\lambda^p(x)>1$ holds for any $\lambda \in (0,\lambda_-^p(x))$.
It follows from Theorem~\ref{thm:lyap_conti} that for any $\lambda \in (0,\lambda_-^p(x))$,
\begin{align*}
	\lim_{\delta \searrow 0} \lim_{n \to \infty} \frac{\alpha_\lambda^{p_n}(x)-\alpha_{\lambda-\delta}^{p_n}(x)}{\delta}
	= \partial_- \alpha_\lambda^p(x)>1.
\end{align*}
This, together with the concavity of $\alpha_\lambda^{p_n}(x)$, yields that if $\lambda \in (0,\lambda_-^p(x))$, then for all small $\delta>0$ there exists $N=N_{\lambda,\delta} \in \N$ such that for all $n \geq N$,
\begin{align*}
	\partial_-\alpha_{\lambda-\delta}^{p_n}(x) \geq \frac{\alpha_\lambda^{p_n}(x)-\alpha_{\lambda-\delta}^{p_n}(x)}{\delta}>1,
\end{align*}
which proves that $\liminf_{n \to \infty} \lambda_-^{p_n}(x) \geq \lambda-\delta$.
Hence, \eqref{eq:limL1} follows by letting $\delta \searrow 0$ and $\lambda \nearrow \lambda_-^p(x)$.

We next treat \eqref{eq:limL2}.
Similarly to the above, if $\lambda>\lambda_+^p(x)$, then there exist $\delta>0$ and $N'=N'_{\lambda,\delta} \in \N$ such that for all $n \geq N'$,
\begin{align*}
	\partial_- \alpha_\lambda^{p_n}(x) \leq \frac{\alpha_\lambda^{p_n}(x)-\alpha_{\lambda-\delta}^{p_n}(x)}{\delta}<1.
\end{align*}
Accordingly, $\limsup_{n \to \infty} \lambda_+^{p_n}(x) \leq \lambda$ holds for all $\lambda>\lambda_+^p(x)$, and \eqref{eq:limL2} follows by letting $\lambda \searrow \lambda_+^p(x)$.
\end{proof}

We finally discuss the finiteness of $\lambda_+^p(x)$ and observe a directional continuity of the rate function uniformly in the law of the percolation configuration.

\begin{lem}\label{lem:interior}
We have $\lambda_+^p(x)<\infty$ for all $x \in (\mathcal{D}_{I^p})^o$.
\end{lem}
\begin{proof}
Letting $x \in (\mathcal{D}_{I^p})^o$, one has $\mu^p(x)<1$ from Corollary~\ref{cor:shape}.
Hence, Proposition~\ref{prop:relation} shows that there exists $\lambda>0$ such that $\alpha^p_\lambda(x)/\lambda<1$, which implies
\begin{align*}
	\partial_-\alpha_\lambda^p(x)
	\leq \frac{\alpha_\lambda^p(x)-\alpha_0^p(x)}{\lambda}
	< 1-\frac{\alpha_0^p(x)}{\lambda}<1.
\end{align*}
This means that $\lambda_+^p(x) \leq \lambda<\infty$.
\end{proof}

\begin{lem}\label{lem:Lipschitz}
If $x \in (\mathcal{D}_{I^p})^o$, $p$ and $p_n$, $n \geq 1$, belong to $(p_c,1]$ and $p_n \to p$ as $n \to \infty$, then there exists $L>0$ (which is independent of $p_n$'s) such that for all $\delta \in (0,1)$,
\begin{align*}
	\limsup_{n \to \infty} I^{p_n}(x) \leq \limsup_{n \to \infty} I^{p_n}(\delta x)+L(1-\delta).
\end{align*}
\end{lem}
\begin{proof}
Fix $x \in (\mathcal{D}_{I^p})^o$.
Thanks to Corollary~\ref{cor:shape} and Lemma~\ref{lem:interior}, there exists $\eta_0,\lambda_0>0$ (which are independent of $p_n$'s) such that $\alpha_{\lambda_0}^p((1+\eta_0)x)/\lambda_0<1$.
From Theorem~\ref{thm:lyap_conti},
\begin{align*}
	\lim_{n \to \infty}\frac{\alpha_{\lambda_0}^{p_n}((1+\eta_0)x)}{\lambda_0}=\frac{\alpha_{\lambda_0}^p((1+\eta_0)x)}{\lambda_0}<1,
\end{align*}
and it follows that for all large $n$,
\begin{align*}
	\frac{\alpha_{\lambda_0}^{p_n}((1+\eta_0)x)-\alpha_0^{p_n}((1+\eta_0)x)}{\lambda_0}
	\leq \frac{\alpha_{\lambda_0}^{p_n}((1+\eta_0)x)}{\lambda_0}<1.
\end{align*}
This, combined with the concavity of $\alpha_\lambda^{p_n}((1+\eta_0)x)$ in $\lambda$, implies that $(1+\eta_0)x \in \mathcal{D}_{I^{p_n}}$ holds for all large $n$.
Note that $I^{p_n}$ is convex owing to the analogous property of the Lyapunov exponent, and so is $\mathcal{D}_{I^{p_n}}$.
Due to $0 \in \mathcal{D}_{I^{p_n}}$, one has $\eta x \in \mathcal{D}_{I^{p_n}}$ for all $\eta \in [0,1+\eta_0]$.
Denote $u:=(1+\eta_0/2)x \in \mathcal{D}_{I^{p_n}}$ and set $\beta:=\eta_0/\{ 2(1-\delta) \}$ for $\delta \in (0,1)$.
We then have $u=x+\beta(1-\delta)x$, or equivalently
\begin{align*}
	x=\frac{1}{1+\beta}u+\frac{\beta}{1+\beta	}\delta x.
\end{align*}
It follows from the convexity of $I^{p_n}$ that
\begin{align*}
	I^{p_n}(x) \leq \frac{1}{1+\beta}I^{p_n}(u)+\frac{\beta}{1+\beta}I^{p_n}(\delta x).
\end{align*}
This, together with Corollary~\ref{cor:shape}, implies
\begin{align*}
	I^{p_n}(x)-I^{p_n}(\delta x)
	\leq \frac{1}{1+\beta}I^{p_n}(u)
	\leq \frac{2\log(2d)}{\eta_0}(1-\delta),
\end{align*}
and letting $n \to \infty$ proves the lemma.
\end{proof}

After the preparation above, let us prove Theorem~\ref{thm:rate_conti}.

\begin{proof}[\bf Proof of Theorem~\ref{thm:rate_conti}]
It is clear from the definition of the rate function and Theorem~\ref{thm:lyap_conti} that for all $x \in \R^d$,
\begin{align*}
	\liminf_{n \to \infty} I^{p_n}(x) \geq I^p(x).
\end{align*}

Our task is now to prove that for $x \in (\mathcal{D}_{I^p})^o$,
\begin{align}\label{eq:rate_upper}
	\limsup_{n \to \infty}I^{p_n}(x) \leq I^p(x).
\end{align}
The proof is divided into two cases: $\lambda_-^p(x)>0$ and $\lambda_-^p(x)=0$.
We first treat the case $\lambda_-^p(x)>0$.
Fix $\lambda' \in (0,\lambda_-^p(x))$ and $\delta \in (1/\partial_-\alpha_{\lambda'}^p(x),1)$ (since $\partial_- \alpha_{\lambda'}^p(x)>1$, one can take such a $\delta$).
Then,
\begin{align*}
	\partial_- \alpha_{\lambda'}^p(\delta x)
	= \delta \partial_- \alpha_{\lambda'}^p(x)
	> 1.
\end{align*}
On the other hand, $\partial_- \alpha_\lambda^p(\delta x)<1$ holds for all $\lambda>\lambda_-^p(x)$.
With these observations, we obtain $\lambda_-^p(\delta x) \geq \lambda'$ and $\lambda_+^p(\delta x) \leq \lambda_-^p(x) \leq \lambda_+^p(x)<\infty$.
Since $\delta x \in (\mathcal{D}_{I^p})^o$, Lemma~\ref{lem:lambda} shows that
\begin{align*}
	\liminf_{n \to \infty} \lambda_-^{p_n}(\delta x) \geq \lambda_-^p(\delta x) \geq \lambda'
\end{align*}
and
\begin{align*}
	\limsup_{n \to \infty} \lambda_+^{p_n}(\delta x) \leq \lambda_+^p(\delta x) \leq \lambda_-^p(x).
\end{align*}
Therefore, for all $\lambda>\lambda_-^p(x)$,
\begin{align*}
	\limsup_{n \to \infty}I^{p_n}(\delta x)
	&\leq \delta \limsup_{n \to \infty}\bigl( \alpha_{\lambda_+^{p_n}(\delta x)}^{p_n}(x)-\lambda_+^{p_n}(\delta x) \bigr)\\
	&\leq \delta \Bigl( \alpha_\lambda^p(x)-\liminf_{n \to \infty} \lambda_-^{p_n}(\delta x) \Bigr)\\
	&\leq \delta (\alpha_\lambda^p(x)-\lambda').
\end{align*}
Consequently, Lemma~\ref{lem:Lipschitz} proves that for all $\delta \in (1/\partial_- \alpha_{\lambda'}^p(x),1)$ and $\lambda>\lambda_-^p(x)$,
\begin{align*}
	\limsup_{n \to \infty}I^{p_n}(x)
	&\leq \limsup_{n \to \infty}I^{p_n}(\delta x)+L(1-\delta)\\
	&\leq \delta (\alpha_\lambda^p(x)-\lambda')+L(1-\delta).
\end{align*}
Letting $\lambda \searrow \lambda_-^p(x)$ and $\delta \nearrow 1$, one has
\begin{align*}
	\limsup_{n \to \infty}I^{p_n}(x) \leq \alpha_{\lambda_-^p(x)}^p(x)-\lambda'.
\end{align*}
Since $\lambda' \in (0,\lambda_-^p(x))$ is arbitrary, Lemma~\ref{lem:attain} gives 
\begin{align*}
	\limsup_{n \to \infty}I^{p_n}(x) \leq \alpha_{\lambda_-^p(x)}^p(x)-\lambda_-^p(x)=I^p(x),
\end{align*}
and \eqref{eq:rate_upper} follows.

We next treat the case $\lambda_-^p(x)=0$.
Then, $\partial_- \alpha_\lambda^p(x) \leq 1$ holds for all $\lambda>0$.
Hence, for $\lambda>0$ and $\delta \in (0,1)$,
\begin{align*}
	\partial_- \alpha_\lambda^p(\delta x)=\delta \partial_- \alpha_\lambda^p(x) \leq \delta <1.
\end{align*}
This means $\lambda_+^p(\delta x)=0$, and Lemma~\ref{lem:lambda} implies that $\lim_{n \to \infty} \lambda_+^{p_n}(\delta x)=0$.
Therefore, for all $\lambda>0$,
\begin{align*}
	\limsup_{n \to \infty}I^{p_n}(\delta x)
	&\leq \limsup_{n \to \infty}\bigl( \alpha_\lambda^{p_n}(\delta x)-\lambda_+^{p_n}(\delta x) \bigr)\\
	&\leq \alpha_\lambda^p(\delta x)-\liminf_{n \to \infty} \lambda_+^{p_n}(\delta x)\\
	&\leq \alpha_\lambda^p(\delta x).
\end{align*}
Since $\alpha_0^p(x)=I^p(x)$, \eqref{eq:rate_upper} follows from the same argument as in the case $\lambda_-^p(x)>0$, and the proof is complete.
\end{proof}


\begin{thebibliography}{10}

\bibitem{Ahl15}
D.~Ahlberg.
\newblock A {H}su--{R}obbins--{E}rd{\H{o}}s strong law in first-passage
  percolation.
\newblock {\em The Annals of Probability}, 43(4):1992--2025, 2015.

\bibitem{AhnPet16}
S.~W. Ahn and J.~Peterson.
\newblock {Oscillations of quenched slowdown asymptotics for ballistic
  one-dimensional random walk in a random environment}.
\newblock {\em Electron. J. Probab.}, 21:27 pp., 2016.

\bibitem{BerMukOka16}
N.~Berger, C.~Mukherjee, and K.~Okamura.
\newblock Quenched large deviations for simple random walks on percolation
  clusters including long-range correlations.
\newblock {\em Comm. Math. Phys.}, 358(2):633--673, 2018.

\bibitem{ComFukNakYos15}
F.~Comets, R.~Fukushima, S.~Nakajima, and N.~Yoshida.
\newblock Limiting results for the free energy of directed polymers in random
  environment with unbounded jumps.
\newblock {\em Journal of Statistical Physics}, 161(3):577--597, 2015.

\bibitem{ComGanZei00}
F.~Comets, N.~Gantert, and O.~Zeitouni.
\newblock Quenched, annealed and functional large deviations for
  one-dimensional random walk in random environment.
\newblock {\em Probability theory and related fields}, 118(1):65--114, 2000.

\bibitem{Cox80}
J.~T. Cox.
\newblock The time constant of first-passage percolation on the square lattice.
\newblock {\em Advances in Applied Probability}, pages 864--879, 1980.

\bibitem{CoxKes81}
J.~T. Cox and H.~Kesten.
\newblock On the continuity of the time constant of first-passage percolation.
\newblock {\em Journal of Applied Probability}, pages 809--819, 1981.

\bibitem{DemPerZei96}
A.~Dembo, Y.~Peres, and O.~Zeitouni.
\newblock Tail estimates for one-dimensional random walk in random environment.
\newblock {\em Communications in mathematical physics}, 181(3):667--683, 1996.

\bibitem{GanZei98}
N.~Gantert and O.~Zeitouni.
\newblock Quenched sub-exponential tail estimates for one-dimensional random
  walk in random environment.
\newblock {\em Communications in mathematical physics}, 194(1):177--190, 1998.

\bibitem{GarMar04}
O.~Garet and R.~Marchand.
\newblock Asymptotic shape for the chemical distance and first-passage
  percolation on the infinite {B}ernoulli cluster.
\newblock {\em ESAIM: Probability and Statistics}, 8:169--199, 2004.

\bibitem{GarMar10}
O.~Garet and R.~Marchand.
\newblock Moderate deviations for the chemical distance in {B}ernoulli
  percolation.
\newblock {\em Alea}, 7:171--191, 2010.

\bibitem{GarMarProThe17}
O.~Garet, R.~Marchand, E.~B. Procaccia, and M.~Th{\'e}ret.
\newblock Continuity of the time and isoperimetric constants in supercritical
  percolation.
\newblock {\em Electronic Journal of Probability}, 22, 2017.

\bibitem{GredHol94}
A.~Greven and F.~den Hollander.
\newblock Large deviations for a random walk in random environment.
\newblock {\em The Annals of Probability}, pages 1381--1428, 1994.

\bibitem{Gri99_book}
G.~Grimmett.
\newblock {\em Percolation}, volume 321.
\newblock Springer Science \& Business Media, 1999.

\bibitem{Kub12b}
N.~Kubota.
\newblock Large deviations for simple random walk on supercritical percolation
  clusters.
\newblock {\em Kodai Mathematical Journal}, 35(3):560--575, 2012.

\bibitem{Kum14_book}
T.~Kumagai.
\newblock {\em Random walks on disordered media and their scaling limits}.
\newblock Springer, 2014.

\bibitem{Le17}
T.~T.~H. Le.
\newblock On the continuity of lyapunov exponents of random walk in random
  potential.
\newblock {\em Bernoulli}, 23(1):522--538, 2017.

\bibitem{Lig99_book}
T.~M. Liggett.
\newblock {\em Stochastic Interacting Systems: Contact, Voter and Exclusion
  Processes}, volume 324.
\newblock Springer Science \& Business Media, 1999.

\bibitem{Nak18}
S.~Nakajima.
\newblock Concentration results for directed polymer with unbounded jumps.
\newblock {\em ALEA}, 15:1--20, 2018.

\bibitem{Pet08_thesis}
J.~Peterson.
\newblock {\em Limiting distributions and large deviations for random walks in
  random environments}.
\newblock ProQuest, 2008.

\bibitem{Pet10}
J.~Peterson.
\newblock Systems of one-dimensional random walks in a common random
  environment.
\newblock {\em Electronic Journal of Probability}, 15:1024--1040, 2010.

\bibitem{PetZei09b}
J.~Peterson and O.~Zeitouni.
\newblock On the annealed large deviation rate function for a multi-dimensional
  random walk in random environment.
\newblock {\em Alea}, 6:349--368, 2009.

\bibitem{PisPov99}
A.~Pisztora and T.~Povel.
\newblock Large deviation principle for random walk in a quenched random
  environment in the low speed regime.
\newblock {\em The Annals of Probability}, 27(3):1389--1413, 1999.

\bibitem{PisPovZei99}
A.~Pisztora, T.~Povel, and O.~Zeitouni.
\newblock Precise large deviation estimates for a one-dimensional random walk
  in a random environment.
\newblock {\em Probability theory and related fields}, 113(2):191--219, 1999.

\bibitem{RA04}
F.~Rassoul-Agha.
\newblock Large deviations for random walks in a mixing random environment and
  other (non-markov) random walks.
\newblock {\em Communications on pure and applied mathematics},
  57(9):1178--1196, 2004.

\bibitem{RASep11}
F.~Rassoul-Agha and T.~Sepp{\"a}l{\"a}inen.
\newblock Process-level quenched large deviations for random walk in random
  environment.
\newblock In {\em Annales de l'institut Henri Poincar{\'e} (B)}, volume~47,
  pages 214--242, 2011.

\bibitem{RASep15_book}
F.~Rassoul-Agha and T.~Sepp{\"a}l{\"a}inen.
\newblock {\em A course on large deviations with an introduction to Gibbs
  measures}, volume 162.
\newblock American Mathematical Soc., 2015.

\bibitem{RASepYil13}
F.~Rassoul-Agha, T.~Sepp{\"a}l{\"a}inen, and A.~Yilmaz.
\newblock Quenched free energy and large deviations for random walks in random
  potentials.
\newblock {\em Communications on Pure and Applied Mathematics}, 66(2):202--244,
  2013.

\bibitem{Var03}
S.~S. Varadhan.
\newblock Large deviations for random walks in a random environment.
\newblock {\em Communications on Pure and Applied mathematics},
  56(8):1222--1245, 2003.

\bibitem{Yil11}
A.~Yilmaz.
\newblock Equality of averaged and quenched large deviations for random walks
  in random environments in dimensions four and higher.
\newblock {\em Probability theory and related fields}, 149(3-4):463--491, 2011.

\bibitem{YilZei10}
A.~Yilmaz and O.~Zeitouni.
\newblock Differing averaged and quenched large deviations for random walks in
  random environments in dimensions two and three.
\newblock {\em Communications in Mathematical Physics}, 300(1):243--271, 2010.

\bibitem{Zei04_book}
O.~Zeitouni.
\newblock Random walks in random environment.
\newblock In {\em Lectures on probability theory and statistics}, pages
  189--312. Springer, 2004.

\bibitem{Zer98a}
M.~P. Zerner.
\newblock {D}irectional decay of the {G}reen's function for a random
  nonnegative potential on {$\mathbf{Z}^d$}.
\newblock {\em Annals of Applied Probability}, pages 246--280, 1998.

\bibitem{Zer98b}
M.~P. Zerner.
\newblock {L}yapounov exponents and quenched large deviations for
  multidimensional random walk in random environment.
\newblock {\em The Annals of Probability}, 26(4):1446--1476, 1998.

\end{thebibliography}

\end{document}